\documentclass[12pt]{amsart}
\usepackage[dvips]{graphicx,color}
\usepackage{amscd,amssymb,amsmath,amsthm,epsfig,xypic}
\usepackage{psfrag}
\usepackage[mathscr]{euscript}
\input xy
\xyoption{all}
\DeclareFontFamily{OT1}{rsfs}{}
\DeclareFontShape{OT1}{rsfs}{n}{it}{<-> rsfs10}{}
\DeclareMathAlphabet{\curly}{OT1}{rsfs}{n}{it}

\newtheorem{thm}{Theorem}

\newtheorem{lemma}[thm]{Lemma}
\newtheorem{prop}[thm]{Proposition}

\newcommand\AAA{\mu}
\newcommand\CCC{\lambda}

\newcommand\C{\mathbb C}

\newcommand\into{\hookrightarrow}
\newcommand\Into{\ar@{^{ (}->}[r]}

\newfont{\bigtimesfont}{cmsy10 scaled \magstep5}
\newcommand{\bigtimes}{\mathop{\lower0.9ex\hbox{\bigtimesfont\symbol2}}}

\newcommand\Spec{\operatorname{Spec}}

\newcommand\Sym{\operatorname{Sym}}

\theoremstyle{remark}

\begin{document}

\title{Almost closed 1-forms}
\author[Pandharipande, Thomas]{R. Pandharipande and R. P. Thomas}
\maketitle

 \begin{abstract}
We construct an algebraic almost closed 1-form with zero scheme
not expressible (even locally) 
as the critical locus of a holomorphic function
on a nonsingular variety. The result answers a question of
Behrend-Fantechi. We correct here an error in our paper
\cite{MPT} where an incorrect construction with the same claimed
properties was proposed.
 \end{abstract}

 \setcounter{section}{-1}
 \thispagestyle{empty}

 \baselineskip=15pt


\section{Introduction} \label{kkxx33}

An algebraic 1-form $\omega$ on a nonsingular variety $V$
is {\em almost closed} if
$d\omega$ vanishes on the zero scheme 
$\mathcal{Z}(\omega) \subset V$ of $\omega$.
For example, for every nonconstant
polynomial $F\in \C[x,y]$, the 1-form
\begin{equation}\label{gbbt}
\omega = dF + F dz
\end{equation}
is almost closed, but not closed,
on $\C^3$ with coordinates $x,y,z$.
By a construction of Behrend-Fantechi \cite{BF2}, the zero scheme
$\mathcal{Z}(\omega)$ carries a natural symmetric
obstruction theory in case $\omega$ is almost closed.

A question asked by Behrend and Fantechi \cite{bdt,BF2} is the following.
If $\omega$ is almost closed, can we find an analytic
neighbourhood 
$$\mathcal{Z} \subset U \subset V$$ and a holomorphic
function 
$\Phi$ on $U$ so the equality
$$\mathcal{Z}(\omega) = \mathcal{Z}(d\Phi)$$
holds as subschemes? In other words, is $\mathcal{Z}(\omega)$
always the critical locus of a holomorphic function? 
In example \eqref{gbbt}, the function
\begin{equation} \label{Joyce}
\Phi= e^zF
\end{equation}
 provides an affirmative answer. Such a $\Phi$ is
called a {\em potential} for $\mathcal{Z}$.

We will construct a nonreduced point at the origin
$$\mathcal{Z} \subset \C^2$$ which is the zero 
scheme of an algebraic almost closed 1-form on a
Zariski open neighbourhood of the orgin in $\C^2$,
but \emph{not} 
the critical locus of any holomorphic function defined
in a neighbourhood of the origin in $\C^2$. 
 We also show $\mathcal{Z}$ cannot be a critical locus 
of a holomorphic function defined locally on $\C^n$
for any $n$.

We will primarily study almost closed 1-forms via formal power series 
analysis at the origin. In Section \ref{Salg} we show how these methods also yield \emph{algebraic} almost closed 1-forms on Zariski open neighbourhoods
of the origin, as relevant to the Behrend-Fantechi question. \medskip

Recently Pantev, To\"en, Vaquie and Vezzosi have introduced the notion of a \emph{$(-1)$-shifted symplectic structure} on a derived scheme. This is a stronger condition than the existence of a symmetric obstruction theory on the underlying scheme $\mathcal Z$. Brav, Bussi, Dupont and Joyce \cite{Joyce} have announced that a scheme $\mathcal Z$ is locally a critical locus if and only if it underlies a derived scheme with $(-1)$-shifted symplectic structure. \medskip

\noindent\textbf{Acknowledgements.} 
We would like to thank Dominic Joyce for showing us the potential \eqref{Joyce} for the zeros of the 1-form \eqref{gbbt} and pointing out the error in \cite[Appendix A.5]{MPT}. Specifically, the averaging argument in
Corollary 34 is incorrect. The error
there does not affect the other counterexamples in \cite[Appendix A]{MPT}.

\section{Construction of the almost closed 1-form} \label{S1}
Consider $\C^2$ with coordinates $x$ and $y$.
We will use capital letters 
 to denote elements
of the ring $\C[\![x,y]\!]$.
In case the element is polynomial and homogenous, 
the degree is designated in the subscript.
As usual partial derivatives will also be denoted by subscripts, so
$$A_{d,xy} = \frac{\partial^2 A_d}{\partial x \partial y} $$
is a homogeneous polynomial of degree $d-2$.

We start by constructing an almost closed 1-form $\sigma$ on $\C^2$ via formal power series.
The construction depends upon the initial choice of a sufficiently 
large degree $d$. 
Let
\begin{eqnarray} 
\sigma &=& A dx+B dy \nonumber \\ \nonumber
&=& (A_d+A_{d+1}+A_{d+2}+\ldots)dx+(B_d+B_{d+1}+B_{d+2}+\ldots)dy,
\end{eqnarray}
be a 1-form starting at order $d$ satisfying
\begin{eqnarray}
A_y-B_x &=& CA+DB\nonumber \\
&=&  (C_0+C_1+C_2+\ldots)(A_d+A_{d+1}+A_{d+2}+\ldots) \label{ac} \\
&& \hspace{-9pt} +(D_0+D_1+D_2+\ldots)(B_d+B_{d+1}+B_{d+2}+\ldots), \nonumber
\end{eqnarray}
where $C,D\in \C[\![x,y]\!]$.

The zero scheme $\mathcal{Z}(\sigma)$ is  $\Spec\frac{\C[\![x,y]\!]}{(A,B)}$.
We shall often use the trivialisation 
$dy\wedge dx$ of the 2-forms on $\C^2$ to identify $d\sigma$ with the 
function $A_y-B_x$. 
Equation \eqref{ac} is the almost closed condition, 
$$d\sigma=A_y-B_x\in(A,B).$$
For sufficient large $d$,
we will prove the general $\sigma$ satisfying \eqref{ac}
has zero scheme which
is not the critical locus of a holomorphic function.

Fixing $C,D\in \C[\![x,y]\!]$,
we construct solutions $\sigma$ to the almost
closed equation \eqref{ac} by the following procedure:
\begin{itemize}
\item First, $A_d$ and $B_d$ are degree $d$ homogeneous polynomials satisfying
\begin{equation} \label{Pintegrable}
A_{d,y}-B_{d,x}=0.
\end{equation}
Thus
\begin{equation}\label{gvv3}
A_d=P_{d+1,x}\,, \quad B_d=P_{d+1,y}\,,
\end{equation} 
for any homogeneous degree $d+1$ polynomial $P_{d+1}$.
We pick a \emph{general} solution: a general element of the vector space
$$
\qquad\xymatrix@C=40pt{\ker\Big(\Sym^d(\C^2)^*\oplus\Sym^d(\C^2)^* \ar[r]^(.6){\partial_y\oplus-
\partial_x} & \Sym^{d-1}(\C^2)^*\Big).}
$$

\item Next, $A_{d+1}$ and $B_{d+1}$ are homogeneous polynomials
of degree $d+1$ satisfying
\begin{equation} \label{Aintegrable}
\quad A_{d+1,y}-B_{d+1,x}=C_0A_d+D_0B_d
\end{equation}
for the given constants $C_0,D_0$. The above is an affine linear 
equation on the vector space of pairs $(A_{d+1},B_{d+1})$ with 
a nonempty set of solutions (for example, 
set $B_{d+1}=0$ and integrate to get $A_{d+1}$).
 We pick a {\em general} element of the affine space of solutions.
\vspace{0pt}

\item At the $k^{th}$ order, we have defined $A_d,\ldots,A_{k-1}$ and 
$B_d,\ldots,B_{k-1}$, and we pick degree $k$ homogeneous polynomials
$A_k,B_k$ satisfying the affine linear equation 
\begin{equation}
\label{mss}
A_{k,y}-B_{k,x}\,=\, \sum_{i=0}^{k-1-d} C_i A_{k-1-i}+ \sum_{i=0}^{k-1-d}
D_i B_{k-1-i} \,.
\end{equation} 
Again the affine space of solutions is nonempty, and we pick a {\em general}
 element.
\end{itemize}
\vspace{8pt}

Thus, we obtain a general formal power series almost-closed 1-form
$\sigma$ starting at order $d$ with given $C,D\in \C[\![x,y]\!]$. 
Later, we will show $\sigma$ can be modified to be algebraic. \medskip

\section{Convergence}
Though not necessary for our construction,
 the power series solution $\sigma$
can be taken to be convergent when conditions\footnote{Stronger results are certainly possible, but the conditions we impose are sufficiently mild that the nonexistence result of Theorem \ref{Zurich} still holds if we replace the word ``general" by ``general within the Euclidean open set of Lemma \ref{CD0}".} are placed on the series
$C,D\in \C[\![x,y]\!]$.

\begin{lemma} \label{CD0}
Suppose  $C_{\ge2}=0=D_{\ge2}$. Then, there is a power series 
solution  $\sigma$ starting at order $d$
 with non-zero radius of convergence.
Moreover, the set of solutions of \eqref{ac} which converge
contains an Euclidean open set in the space of all solutions.
%
\end{lemma}

\begin{proof}
When $C_{\ge2}=0=D_{\ge2}$, equation \eqref{mss} simplifies to
\begin{equation} \label{msss}
A_{k,y}-B_{k,x}\,=\,C_0A_{k-1}+D_0B_{k-1}+C_1A_{k-2}+D_1B_{k-2}.
\end{equation}
Let $\|p\|$ denote the maximum of the absolute values of the coefficients of the polynomial $p(x,y)$, and make the following definitions:
\begin{eqnarray*}
\CCC &=& \max(|C_0|,|D_0|,2\|C_1\|,2\|D_1\|), \\
\AAA_k &=& \max(\|A_k\|,\|B_k\|,\|A_{k-1}\|,\|B_{k-1}\|),\\
E_{k-1} &=& C_0A_{k-1}+D_0B_{k-1}+C_1A_{k-2}+D_1B_{k-2}. 
\end{eqnarray*}
We easily find
\begin{align*}
\|E_{k-1}\|\le &
\ |C_0|\|A_{k-1}\|+|D_0|\|B_{k-1}\|+2\|C_1\|\|A_{k-2}\|+2\|D_1\|\|B_{k-2}\|
\\ \le &\ 4\CCC\AAA_{k-1}.
\end{align*}
A general solution of \eqref{msss} is provided by integrating
\begin{eqnarray*}
A_{k,y} &=& F_{k-1}, \\
B_{k,x} &=& F_{k-1}-E_{k-1},
\end{eqnarray*}
for a general choice of homogeneous degree $k-1$ polynomial $F_{k-1}$.
We choose now $F_{k-1}$ to satisfy
$$\|F_{k-1}\|<4\CCC\AAA_{k-1},$$
which is a Euclidean open condition.

Since integration divides all coefficients by integers we see 
\begin{equation} \label{estimate}
\|A_k\|<4\CCC\AAA_{k-1},\quad \|B_k\|<8\CCC\AAA_{k-1},
\end{equation}
so
$$
\AAA_k\le\max(1,8\CCC)\AAA_{k-1}.
$$
Therefore
$$
\AAA_{k}\le\alpha\beta^k
$$
for some constants $\alpha,\beta$. The 1-form $\sigma$ solving \eqref{ac}
 thus has radius of convergence at least $\beta^{-1}$.
\end{proof}

\section{Algebraic almost closed 1-forms} \label{Salg}

We return to the formal power series solutions $\sigma$ of \eqref{ac}
starting at degree $d$. 
For general choices of $A_d$ and $B_d$, the zero locus $\mathcal{Z}(\sigma)$ 
will be a 
nonreduced scheme $\mathcal{Z}_0(\sigma)$
at the origin of $\C^2$ (plus possibly some other 
disjoint loci). For sufficiently large $N$, $$\mathcal{Z}_0(\sigma) \subset
\text{Spec}(\C[\![x,y]\!]/\mathfrak m^N),$$ 
where $\mathfrak m=(x,y)$ is the maximal ideal of the origin. Equivalently,
\begin{equation} \label{N}
\mathfrak m^N\subset(A,B)\subset\C[\![x,y]\!].
\end{equation}
Consider the polynomial 1-form
\begin{eqnarray*}
\sigma_{\le N} &=& A_{\le N}\,dx+B_{\le N}\,dy \\
&=& (A_d+A_{d+1}+\ldots+A_N)dx+(B_d+B_{d+1}+\ldots+B_N)dy.
\end{eqnarray*}
While $\sigma_{\le N}$
 need not be almost closed on all of $\C^2$ (issues may arise
at the zeros of $\sigma$ away from the origin),
 $\sigma_{\le N}$ \emph{is} almost closed
in a Zariski open neighbourhood of the origin:

\begin{lemma}
In the localised ring $\C[x,y]_{(x,y)}$, the 1-form $\sigma_{\le N}$ is almost closed,
$$(d\sigma_{\le N})\subset (A_{\le N}, B_{\le N})\subset\C[x,y]_{(x,y)}.$$
\end{lemma}

\begin{proof}
First, we claim $\sigma_{\le N}$ is almost closed when considered in the ring $\C[\![x,y]\!]$ of formal power series at the origin of $\C^2$. 
In fact, $d\sigma_{\le N}$ is
\begin{eqnarray}
A_{\le N,y}-B_{\le N,x} &=& (A_y-B_x)_{\le N-1} \nonumber \\
&=& (CA+DB)_{\le N-1} \nonumber \\
&=& CA_{\le N}+DB_{\le N}+\epsilon(N), \label{fg}
\end{eqnarray}
where $\epsilon(N)$ consists only of terms of degree $\ge N$ and 
therefore 
$$\epsilon(N) \in \mathfrak m^N\subset(A,B)$$ by \eqref{N}.
 So we see  $\sigma_{\le N}$ is almost closed modulo $\mathfrak m^N$, and 
we can write
\begin{eqnarray}
\epsilon(N) &=& C^0A+D^0B \nonumber \\
&=& C^0A_{\le N}+D^0B_{\le N}+(C^0A_{>N}+D^0B_{>N})  \nonumber \\
&=& C^0A_{\le N}+D^0B_{\le N}+\epsilon(N+1) \label{eps}
\end{eqnarray}
for series $C^0,D^0 \in \C[\![x,y]\!]$.
Here
$$
\epsilon(N+1)\in\mathfrak m^{N+1}=\mathfrak m.\mathfrak m^N\subset\mathfrak m.(A,B),
$$
which can therefore be written 
\begin{equation} \label{eps'}
\epsilon(N+1)\,=\,C^1A_{\le N}+D^1B_{\le N}+\epsilon(N+2)
\end{equation}
just as in \eqref{eps}, with $C^1,D^1\in\mathfrak m$. 
Continuing inductively, we write
$$
\mathfrak m^k.(A,B)\supset \mathfrak m^{N+k}\,\ni\ \epsilon(N+k)\,=\,
C^kA_{\le N}+D^kB_{\le N}+\epsilon(N+k+1),
$$
where $C^k,D^k\in\mathfrak m^k$ consist only of terms of degree $\ge k$.
We finally obtain 
$$
A_{\le N,y}-B_{\le N,x}\,=\,\overline{C}A_{\le N}+\overline{D}B_{\le N},
$$
where the series
 $$\overline{C}=C+C^0+C^1+\ldots,\ \ \ \overline{D}=D+D^0+D^1+\ldots$$
are convergent in $\C[\![x,y]\!]$. 
So $\sigma_{\le N}$ is indeed almost closed in $\C[\![x,y]\!]$. \medskip

The above results can be written as the vanishing
$$
[A_{\le N,y}-B_{\le N,x}]=0 \quad\text{in the }\C[\![x,y]\!]\text{-module}\ \frac{\C[\![x,y]\!]}{(A_{\le N},B_{\le N})}\,,
$$
or equivalently that
\begin{equation} \label{fp}
\xymatrix@C=18pt{
0 \ar[r]& \C[\![x,y]\!] \ar[r]^1& \C[\![x,y]\!] \ar[rrr]^{A_{\le N,y}-B_{\le N,x}}&&& \frac{\C[\![x,y]\!]}{(A_{\le N},B_{\le N})}}
\end{equation}
is exact. We are left with showing similarly that
\begin{equation} \label{loc}
\xymatrix@C=18pt{
0 \ar[r]& \C[x,y]_{(x,y)} \ar[r]^1& \C[x,y]_{(x,y)} \ar[rrr]^{A_{\le N,y}-B_{\le N,x}}&&& \frac{\C[x,y]_{(x,y)}}{(A_{\le N},B_{\le N})}}
\end{equation}
is also exact.

Now \eqref{loc} pulls back to \eqref{fp} via the inclusion \begin{equation} \label{in}
\C[x,y]_{(x,y)}\into\C[\![x,y]\!].
\end{equation}
Since \eqref{in} is flat \cite[Theorem 8.8]{Mat} and a local map of local rings, it is faithfully flat by \cite[Theorem 7.2]{Mat}. Thus the exactness of \eqref{loc} follows from that of \eqref{fp}.
\end{proof}

\section{Non-existence of a potential} \label{S4}
Consider again formal power series solutions $\sigma$ of \eqref{ac}
starting at degree $d$. 
We will show for almost every 
choice\footnote{In fact, we can set $C=x,\,D=0$ and $B_{\ge(d+2)}=0$ and 
the proof still works. But the
restriction gives no significant simplification in notation.} of $C$ and
$D$,  the zero scheme
$\mathcal{Z}(\sigma)$ is not the critical locus of at the origin of 
any formal function if $A$ and $B$ 
are chosen to be \emph{general} solutions of \eqref{mss}.

\begin{thm} \label{Zurich} If $C_{1,x}+D_{1,y}\ne0$,
then for $d\ge18$ and general choices of $A,B$ satisfying \eqref{mss}, 
there is no
potential function $\Phi\in\C[\![x,y]\!]$ satisfying
 $$(\Phi_x,\Phi_y) =(A,B)\subset\C[\![x,y]\!].$$
\end{thm}

Assume a potential function $\Phi\in\C[\![x,y]\!]$ exists satisfying
$$
(\Phi_x,\Phi_y)=(A,B)\,\subset\C[\![x,y]\!].
$$
Then, $\Phi_x-A$ and $\Phi_y-B$ are both in the ideal $(A,B)$, so we have
\begin{eqnarray}
\Phi_x-A &=& XA+YB, \label{ZZ} \\
\Phi_y-B &=& ZA+WB, \nonumber
\end{eqnarray}
for series $X,Y,Z,W\in \C[\![x,y]\!]$. 
As usual, we write 
$$X=X_0+X_1+\ldots $$
 with $X_k$ homogeneous of degree $k$ (and similarly for
$Y$, $Z$, and $W$). 
To find $\Phi$ satisfying\footnote{Equation \eqref{ZZ} implies only
$(\Phi_x,\Phi_y)\subset(A,B)$ and has the trivial solution 
$$X=1=W \ \ \ \text{and}\ \ \  Y=0=Z$$ corresponding to constant $f$. We will rule out the trivial solution
by requiring  $(\Phi_x,\Phi_y)=(A,B)$ shortly.} \eqref{ZZ}, 
integrability $$\Phi_{xy}=\Phi_{yx},$$ is
a necessary and sufficient condition. In other words, we must have
\begin{equation} \label{eqn}
\quad -A_y+B_x\,=\,(XA+YB)_y-(ZA+WB)_x.
\end{equation}

We will analyse equation \eqref{eqn} for $X,Y,Z,W$ order by order
 modulo higher and higher powers $\mathfrak m^k$ of the maximal ideal. 
At each order, the issue is linear.
 For the first few orders, equation \eqref{eqn} can be easily solved (with
several degrees of freedom).
But each further
stage imposes more stringent conditions on the choices at the previous stage. In degree $d+2$, we will see the conditions become overdetermined, with no nontrivial solutions $X,Y,Z,W$ for general $A,B$ satisfying \eqref{ac}. 
\medskip
\vspace{10pt}

\noindent\textbf{Degree} $\mathbf{d-1}$.
In homogeneous degree $d-1$, \eqref{eqn} yields
$$
-A_{d,y}+B_{d,x}=X_0A_{d,y}+Y_0B_{d,y}-Z_0A_{d,x}-W_0B_{d,x}.
$$
From \eqref{Pintegrable}, we have the relation $A_{d,y}=B_{d,x}$.
 Otherwise the partial derivatives of $A_d$ and $B_d$ are completely general. 
Therefore the resulting equation
$$
0=(X_0-W_0)A_{d,y}+Y_0B_{d,y}-Z_0A_{d,x}
$$
implies the vanishing of $(X_0-W_0)$, $Y_0$, and $Z_0$. 
Thus in homogeneous degree $d$, the equations \eqref{ZZ} become
\begin{eqnarray*}
\Phi_{d+1,x} &=& (1+X_0)A_d \\
\Phi_{d+1,y} &=& (1+X_0)B_d,\!
\end{eqnarray*}
with $\Phi$ having no terms of degree $\le d$. Since we require
 $(\Phi_x,\Phi_y)$ to generate the ideal $(A,B)$
and since $$A_d, B_d\neq 0$$ by generality, we find $1+X_0\neq 0$ . Rescaling $\Phi$, we can assume 
$$1=1+X_0$$
without loss of generality. So we have found the following conditions at order $d-1$:
\begin{equation} \label{d-1g}
X_0=Y_0=Z_0=W_0=0.
\end{equation}
\vspace{10pt}

\noindent\textbf{Degree} $\mathbf{d}$.
In homogeneous degree $d$, \eqref{eqn} yields
\begin{eqnarray*}
-A_{d+1,y}+B_{d+1,x} \!\!\!&=&\!\!\! (X_1A_d)_y+(Y_1B_d)_y-(Z_1A_d)_x-(W_1B_d)_x \\
&& \hspace{-5mm}+(X_0A_{d+1})_y+(Y_0B_{d+1})_y-(Z_0A_{d+1})_x-(W_0B_{d+1})_x.
\end{eqnarray*}
By our work in the previous degree \eqref{d-1g}, the second line 
on the right vanishes identically. After substituting \eqref{Aintegrable} on the left side, we find
\begin{equation} \label{ddd}
-C_0A_d-D_0B_d\,=\,(X_1A_d)_y+(Y_1B_d)_y-(Z_1A_d)_x-(W_1B_d)_x.
\end{equation}
We rewrite \eqref{ddd} as
 the vanishing of
\begin{multline*}
(X_{1,y}-Z_{1,x}+C_0)A_d+X_1A_{d,y}-Z_1A_{d,x}\\
+(Y_{1,y}-W_{1,x}+D_0)B_d+Y_1B_{d,y}
-W_1B_{d,x}.
\end{multline*}
After expanding the linear unknowns out fully via
\begin{eqnarray*}
 X_1=X_{1,x}x+X_{1,y}y, & & \ Y_1=Y_{1,x}x+Y_{1,y}y,\\
 Z_1=Z_{1,x}x+Z_{1,y}y, & & W_1=W_{1,x}x+W_{1,y}y,
\end{eqnarray*}
we obtain a relation among the degree $d$ homogeneous polynomials
\begin{eqnarray}
A_d,\,xA_{d,x},\,xA_{d,y},\,yA_{d,x},\,yA_{d,y}, \label{kww3}\\
B_d,\,xB_{d,x},\,xB_{d,y},\,yB_{d,x},\,yB_{d,y}.\nonumber
\end{eqnarray}
Since $A_d$ and $B_d$ were chosen generically
subject to $A_{d,y} = B_{d,x}$ , the polynomials
\eqref{kww3} are linearly independent except for the relations
\begin{eqnarray}
A_{d,y} &=& B_{d,x} \quad\text{(multiplied by }x,y),\hspace{-3cm} \nonumber \\
xA_{d,x}+yB_{d,y} &=& dA_d, \label{eulerrel} \\
xA_{d,x}+yB_{d,y} &=& dB_d. \nonumber
\end{eqnarray}
The first is \eqref{Pintegrable}. The last two
are the Euler homogeneity relations. A simple
check shows $d\ge5$ is sufficient to achieve the independence. 

Using the first equations
of \eqref{eulerrel} 
to eliminate $xB_{d,x}$ and $yB_{d,x}$ and the last two to eliminate $A_d,B_d$,
 we find the following equations:
\vspace{8pt}
\begin{eqnarray*}
\text{Coefficient of\ }\ xA_{d,x}: && \quad 0\,=\,(X_{1,y}-Z_{1,x}+C_0)/d-Z_{1,x}, \\ yA_{d,x}: && \quad 0\,=\,-Z_{1,y}, \\
xA_{d,y}: && \quad 0\,=\,X_{1,x}+(Y_{1,y}-W_{1,x}+D_0)/d-W_{1,x}, \\
yA_{d,y}: && \quad 0\,=\,(X_{1,y}-Z_{1,x}+C_0)/d+X_{1,y}-W_{1,y}, \\
xB_{d,y}: && \quad 0\,=\,Y_{1,x}, \\
yB_{d,y}: && \quad 0\,=\,(Y_{1,y}-W_{1,x}+D_0)/d+Y_{1,y}.
\end{eqnarray*}
Therefore we find
\begin{eqnarray}
Y_{1,x}&=&0\ =\ Z_{1,y}, \nonumber \\
W_{1,y}&=&X_{1,y}+Z_{1,x}, \nonumber \\
X_{1,x}&=&Y_{1,y}+W_{1,x}, \label{result} \\
C_0&=&(d+1)Z_{1,x}-X_{1,y}, \nonumber \\
D_0&=&W_{1,x}-(d+1)Y_{1,y}. \nonumber 
\end{eqnarray}

%
The equations \eqref{result}
 can be solved with room to spare --- there is a 2 dimensional affine space of solutions. However, the constraints in next degree will 
impose further conditions which specify $X_1,Y_1,Z_1,W_1$ uniquely. \medskip
\vspace{10pt}

\noindent\textbf{Degree} $\mathbf{d+1}$.
In homogeneous degree $d+1$, equation \eqref{eqn} yields
\begin{eqnarray*}
-A_{d+2,y}+B_{d+2,x} \!\!&=&\!\!
(X_2A_d)_y+(Y_2B_d)_y-(Z_2A_d)_x-(W_2B_d)_x \\
&& \hspace{-5mm}+(X_1A_{d+1})_y+(Y_1B_{d+1})_y-(Z_1A_{d+1})_x-(W_1B_{d+1})_x \\
&& \hspace{-5mm}+(X_0A_{d+2})_y+(Y_0B_{d+2})_y-(Z_0A_{d+2})_x-(W_0B_{d+2})_x.
\end{eqnarray*}
By the constraints \eqref{d-1g}, the third line vanishes identically. 
Substituting equation \eqref{mss} for $k=d+2$  on the left side yields
\begin{align} 
-C_0A_{d+1}-D_0&B_{d+1}-C_1A_d-D_1B_d= \nonumber \\
&\qquad (X_2A_d)_y+(Y_2B_d)_y-(Z_2A_d)_x-(W_2B_d)_x \label{d-1} \\
&+(X_1A_{d+1})_y+(Y_1B_{d+1})_y-(Z_1A_{d+1})_x-(W_1B_{d+1})_x. \nonumber
\end{align}

We work first modulo those degree $d+1$ polynomials generated by 
$A_d,B_d$ and their first partial derivatives. More
preciely, let $$V\subset\,\Sym^{d+1}(\C^2)^*$$ 
denote the subspace spanned by $x$ and $y$ multiplied by $A_d,B_d$, and 
$x^2,xy,y^2$ multiplied by $A_{d,x},A_{d,y},
B_{d,x},B_{d,y}$. The subspace $V$
 has dimension 9 due to the relations \eqref{eulerrel}.
In the quotient space $\Sym^{d+1}(\C^2)^*/V$, equation \eqref{d-1} is
\begin{multline} \label{pqq22}
-C_0A_{d+1}-D_0B_{d+1}\,= \\
(X_1A_{d+1})_y+(Y_1B_{d+1})_y-(Z_1A_{d+1})_x-(W_1B_{d+1})_x,
\end{multline}
where all terms are taken mod $V$. 

Equation \eqref{pqq22} has form identical to equation \eqref{ddd}
analysed in the
 previous degree with $A_d,B_d$ replaced by $A_{d+1},B_{d+1}$. 
The analysis of the previous sections applies again here: 
the first relation of \eqref{eulerrel} holds mod $V$ by \eqref{Aintegrable}, and the two Euler relations hold with $d$ replaced by $d+1$. By generality,
for $d\ge13$, the polynomials $A_{d+1},B_{d+1}$ and their partial derivatives (multiplied by $x,y$) are independent of $A_d,B_d$ (multiplied by $x,y$) and their partial derivatives (multiplied by $x^2,xy,y^2$) except for relation \eqref{Aintegrable} and the Euler relations. We conclude the equations
corresponding to \eqref{result} hold (with $d$ replaced by $d+1$):
\begin{eqnarray*}
Y_{1,x}&=&0\ =\ Z_{1,y}, \\
W_{1,y}&=&X_{1,y}+Z_{1,x}, \\
X_{1,x}&=&Y_{1,y}+W_{1,x}, \\
C_0&=&(d+2)Z_{1,x}-X_{1,y}, \\
D_0&=&W_{1,x}-(d+2)Y_{1,y}.
\end{eqnarray*}
After combining with the original equations \eqref{result},
 we find
\begin{eqnarray}
Y_{1,x}&=&0\ =\ Y_{1,y}, \nonumber \\
Z_{1,x}&=&0\ =\ Z_{1,y}, \nonumber \\
X_{1,y}&=&W_{1,y}\ =\ -C_0, \label{c0d0} \\
X_{1,x}&=&W_{1,x}\ =\ D_0. \nonumber
\end{eqnarray}
We have uniquely solved for $X_1,Y_1,Z_1,W_1$. 

We next consider $X_2,Y_2,Z_2,W_2$. There will be no obstruction to solving for $X_2,Y_2,Z_2,W_2$ here.
However, in the next degree, we will find further constraints: the
resulting  
overdetermined system for
$X_2,Y_2,Z_2,W_2$ will
have solutions only if $C_{1,x}+D_{1,y}=0$. \smallskip

After substituting the constraints \eqref{c0d0} into \eqref{d-1},
 we obtained a simpler equation:
\begin{multline*}
-C_1A_d-D_1B_d= (X_2A_d)_y+(Y_2B_d)_y-(Z_2A_d)_x-(W_2B_d)_x\\
+X_1(A_{d+1,y}-B_{d+1,x}).
\end{multline*}
By \eqref{Aintegrable} and \eqref{c0d0}, the last term 
$X_1(A_{d+1,y}-B_{d+1,x})$ is
$$
X_1(C_0A_d+D_0B_d)=(D_0x-C_0y)(C_0A_d+D_0B_d).
$$
We define new terms
 $$\widetilde C_1=C_1+C_0(D_0x-C_0y), \ \ \ \widetilde D_1=D_1+D_0(D_0x-C_0y).$$
 Then, we can write equation \eqref{d-1} as
\begin{equation} \label{tilde}
-\widetilde C_1A_d-\widetilde D_1B_d=(X_2A_d)_y+(Y_2B_d)_y-(Z_2A_d)_x-(W_2B_d)_x.
\end{equation}
Notice the similarity to \eqref{ddd}. 

We solve \eqref{tilde} following our approach to \eqref{ddd}.
After expanding the unknowns,
\begin{eqnarray*}
X_2 &= &X_{2,xx}\frac{x^2}2+X_{2,xy}xy+X_{2,yy}\frac{y^2}2, \\
Y_2 &= &Y_{2,xx}\frac{x^2}2+Y_{2,xy}xy+Y_{2,yy}\frac{y^2}2, \\
Z_2 &= &Z_{2,xx}\frac{x^2}2+Z_{2,xy}xy+Z_{2,yy}\frac{y^2}2, \\
W_2 &= &W_{2,xx}\frac{x^2}2+W_{2,xy}xy+W_{2,yy}\frac{y^2}2, 
\end{eqnarray*}
we obtain 
\begin{align}
-(\tilde C_{1,x}x&+\tilde C_{1,y}y)A_d-(\tilde D_{1,x}x+\tilde D_{1,y}y)B_d= \nonumber \\ & \quad \Big(\!X_{2,xx}\frac{x^2}2+X_{2,xy}xy+
X_{2,yy}\frac{y^2}2\Big)A_{d,y}+(X_{2,xy}x+X_{2,yy}y)A_d \nonumber \\ &
+\Big(Y_{2,xx}\frac{x^2}2+Y_{2,xy}xy+Y_{2,yy}\frac{y^2}2\Big)B_{d,y}+(Y_{2,xy}x+
Y_{2,yy}y)B_d \nonumber \\ & -\Big(Z_{2,xx}\frac{x^2}2+Z_{2,xy}xy+
Z_{2,yy}\frac{y^2}2\Big)A_{d,x}-(Z_{2,xx}x+ Z_{2,xy}y)A_d \nonumber \\ &
-\Big(W_{2,xx}\frac{x^2}2+W_{2,xy}xy+W_{2,yy}\frac{y^2}2\Big)B_{d,x}-
(W_{2,xx}x+W_{2,xy}y)B_d. \label{train}
\end{align}
We consider the above to be a relation among
\begin{align*}
xA_d,yA_d,x^2A_{d,x},xyA_{d,x},y^2A_{d,x},x^2A_{d,y},xyA_{d,y},y^2A_{d,y}, \\ 
xB_d,yB_d,x^2B_{d,x},xyB_{d,x},y^2B_{d,x},x^2B_{d,y},xyB_{d,y},y^2B_{d,y}.
\end{align*}
By generality,
for $d\ge7$, these are linearly independent degree $d+1$ homogeneous polynomials modulo \eqref{Pintegrable} and the Euler relations \eqref{eulerrel}:
\begin{align} \nonumber
A_{d,y}=&\ B_{d,x} \quad  \text{(multiplied by }x^2,xy,y^2), \\ \hspace{1cm}
xA_{d,x}+yA_{d,y}=&\ dA_d\quad \text{(multiplied by }x, y), \label{dr}\\
xB_{d,x}+yB_{d,y}=&\ dB_d\quad \text{(multiplied by }x, y). \nonumber
\end{align}
We use \eqref{dr} to 
eliminate the terms
$$x^2B_{d,x}, xyB_{d,x}, y^2B_{d,x}, xA_d,yA_d,xB_d,yB_d$$
 on the right hand side of  \eqref{train}.
Then indepedence yields
the equations
\begin{eqnarray*}
x^2A_{d,x}: && \quad 0\,=\, \widetilde C_{1,x}/d
-Z_{2,xx}/2+X_{2,xy}/d-Z_{2,xx}/d, \\
xyA_{d,x}: && \quad 0\,=\,\widetilde C_{1,y}/d-Z_{2,xy}+X_{2,yy}/d-Z_{2,xy}/d, \\
y^2A_{d,x}: && \quad 0\,=\,-Z_{2,yy}/2, \\
x^2A_{d,y}: && \quad 0\,=\,\widetilde D_{1,x}/d
+X_{2,xx}/2-W_{2,xx}/2+Y_{2,xy}/d-W_{2,xx}/d, \\
xyA_{d,y}: && \quad 0\,=\,\widetilde C_{1,x}/d+\widetilde D_{1,y}/d
+X_{2,xy}-W_{2,xy}+X_{2,xy}/d
\\ && \hspace{4cm} -Z_{2,xx}/d+Y_{2,yy}/d-W_{2,xy}/d, \\
y^2A_{d,y}: && \quad 0\,=\,\widetilde C_{1,y}/d
+X_{2,yy}/2-W_{2,yy}/2+X_{2,yy}/d-Z_{2,xy}/d, \\
x^2B_{d,y}: && \quad 0\,=\,Y_{2,xx}/2, \\
xyB_{d,y}: && \quad 0\,=\,\widetilde D_{1,x}/d+Y_{2,xy}+Y_{2,xy}/d-W_{2,xx}/d, \\
y^2B_{d,y}: && \quad 0\,=\,\widetilde D_{1,y}/d+Y_{2,yy}/2+Y_{2,yy}/d-W_{2,xy}/d.
\end{eqnarray*}
Use the 2nd equation to eliminate $X_{2,yy}$ from the 6th. Use the 8th equation to eliminate $W_{2,xx}$ from the 4th. Finally,
 use the 1st and last equations to remove $X_{2,xy}$ and $W_{2,xy}$ from the 5th. 
Tidying up, we find $Z_{2,yy}=0=Y_{2,xx}$ and
\begin{eqnarray}
(d+2)Z_{2,xx}-2X_{2,xy} &=& 2\widetilde C_{1,x}, \nonumber \\
(d+1)Z_{2,xy}-X_{2,yy} &=&\widetilde C_{1,y}, \nonumber \\
X_{2,xx}-(d+3)Y_{2,xy} &=& \widetilde D_{1,x}, \nonumber \\
(d+3)(Z_{2,xx}-Y_{2,yy})&=&2(\widetilde C_{1,x}+\tilde D_{1,y}), \label{dv} \\
-W_{2,yy}+(d+3)Z_{2,xy} &=&\widetilde C_{1,y}, \nonumber \\
(d+1)Y_{2,xy}-W_{2,xx} &=&-\widetilde D_{1,x}, \nonumber \\
(d+2)Y_{2,yy}-2W_{2,xy} &=&-2\widetilde D_{1,y}. \nonumber 
\end{eqnarray}
Therefore we can specify $Z_{2,xx},Z_{2,xy},Y_{2xy}$ arbitrarily and uniquely solve for $X_2,Y_2,Z_2,W_2$ after that. However, 
the conditions at the next degree will place futher constraints.\medskip
\vspace{10pt}

\noindent\textbf{Degree} $\mathbf{d+2}$.
In homogeneous degree $d+2$, equation \eqref{eqn} yields
\begin{eqnarray*}
-A_{d+3,y}+B_{d+3,x} \!\!&=&\!\!\!
(X_3A_d)_y+(Y_3B_d)_y-(Z_3A_d)_x-(W_3B_d)_x \\
&& \hspace{-6mm}+(X_2A_{d+1})_y+(Y_2B_{d+1})_y-(Z_2A_{d+1})_x-(W_2B_{d+1})_x \\
&& \hspace{-6mm}+(X_1A_{d+2})_y+(Y_1B_{d+2})_y-(Z_1A_{d+2})_x-(W_1B_{d+2})_x \\
&& \hspace{-6mm}+(X_0A_{d+3})_y+(Y_0B_{d+3})_y-(Z_0A_{d+3})_x-(W_0B_{d+3})_x.
\end{eqnarray*}
By the constraints \eqref{d-1g}, the fourth line vanishes identically. Substituting
\eqref{mss} for $k=d+3$ on the left hand side yields
$$
-C_0A_{d+2}-D_0B_{d+2}-C_1A_{d+1}-D_1B_{d+1}-C_2A_d-D_2B_d.
$$
Work modulo the (12 dimensional space of) degree 
$d+2$ homogeneous polynomials generated by $A_d,B_d$ 
(multiplied by $x^2,xy$ or $y^2$) 
and their first partial derivatives (multiplied by $x^3,x^2y,xy^2,y^3$). We find
\begin{align*}
-C_0A_{d+2}-D_0&B_{d+2}-C_1A_{d+1}-D_1B_{d+1} =  \\
&\hspace{5mm} (X_2A_{d+1})_y+ (Y_2B_{d+1})_y-(Z_2A_{d+1})_x-(W_2B_{d+1})_x \\
&+(X_1A_{d+2})_y+ (Y_1B_{d+2})_y-(Z_1A_{d+2})_x-(W_1B_{d+2})_x.
\end{align*}
Substituting \eqref{c0d0} for $X_1,Y_1,Z_1,W_1$ into the third line and moving it to the first gives
\begin{multline*}
-C_1A_{d+1}-D_1B_{d+1} - (D_0x-C_0y)(A_{d+2,y}-B_{d+2,x}) = \\
(X_2A_{d+1})_y+(Y_2B_{d+1})_y-(Z_2A_{d+1})_x-(W_2B_{d+1})_x.
\end{multline*}
By \eqref{mss} for $k=d+2$, we know the term
 $$A_{d+2,y}-B_{d+2,x} =C_0A_{d+1}+D_0B_{d+1}+C_1A_d+D_1B_d\,.$$
 Since we are working modulo $A_d,B_d$, we obtain
\begin{multline} \label{tilde2}
-\widetilde C_1A_{d+1}-\widetilde D_1B_{d+1}=\\
(X_2A_{d+1})_y+(Y_2B_{d+1})_y-(Z_2A_{d+1})_x-(W_2B_{d+1})_x,
\end{multline}
where 
 $$\widetilde C_1=C_1+C_0(D_0x-C_0y), \ \ \ \widetilde D_1=D_1+D_0(D_0x-C_0y),$$
as before.

Equation \eqref{tilde2} is precisely
the same as equation \eqref{tilde} studied at the previous degree
 with $d$ replaced by $d+1$. Also,
the relations \eqref{dr} still hold (after replacing $d$ by $d+1$) 
by \eqref{Aintegrable}, since we are working mod $A_d,B_d$. And these are 
the only relations which hold for general $A_{d+1},B_{d+1}$ once $d\ge18$.

Hence, the analysis in the previous degree applies here verbatim.
We derive the same equations \eqref{dv} with $d$ replaced by $d+1$.
The first is
$Z_{2,yy}=0=Y_{2,xx}$. And the rest are
\begin{eqnarray*}
(d+3)Z_{2,xx}-2X_{2,xy} &=& 2\widetilde C_{1,x}, \\
(d+2)Z_{2,xy}-X_{2,yy} &=&\widetilde C_{1,y}, \\
X_{2,xx}-(d+4)Y_{2,xy} &=& \widetilde D_{1,x}, \\
(d+4)(Z_{2,xx}-Y_{2,yy})&=&2(\widetilde C_{1,x}+\widetilde D_{1,y}), \\
-W_{2,yy}+(d+4)Z_{2,xy} &=&\widetilde C_{1,y}, \\
(d+2)Y_{2,xy}-W_{2,xx} &=&-\widetilde D_{1,x}, \\
(d+3)Y_{2,yy}-2W_{2,xy} &=&-2\widetilde D_{1,y}.
\end{eqnarray*}
From the central equation and the
 counterpart in \eqref{dv}, we obtain
 a necessary condition for there to exist any solutions:
$$
\widetilde C_{1,x}+\widetilde D_{1,y}=0,
$$
or equivalently,
\begin{equation} \label{cdxy}
C_{1,x}+D_{1,y}=0.
\end{equation}
We have completed the proof of Theorem 3. \qed

\vspace{10pt}

Since $C_1$ and $D_1$ can chosen arbitrarily at the beginning,
choosing them to violate \eqref{cdxy} will imply the non-existence
of a potential $\Phi$. More precisely, when
$A_k,B_k$ are general solutions of the equations \eqref{mss}, 
the resulting almost closed 1-form generates an ideal at the origin which 
is not the critical locus of any formal power series $\Phi\in \mathbb{C}[\![x,y]\!]$. 

By Lemma 2, we can find an algebraic almost closed 1-form
on a Zariski open set of the origin in $\mathbb{C}^2$
which is not the critical locus  of any holomorphic (or even formal) function
$\Phi$ defined near the origin of $\mathbb{C}^2$.
\medskip


\section{Embedding in higher dimensions}

Let $\sigma$ be an algebraic
almost closed 1-form on a Zariski open set of the origin
in $\mathbb{C}^2$ (as constructed above) 
whose zero locus $\mathcal Z(\sigma)$ is both $0$-dimensional and \emph{not}
the critical locus of any holomorphic potential function $\Phi$
near the origin in $\mathbb{C}^2$.

\begin{prop}
The scheme $\mathcal{Z}(\sigma)\subset \C^2$
cannot be written
as the critical locus $\mathcal{Z}(d\Phi)$ of any 
holomorphic function $\Phi$ on 
a nonsingular variety.
\end{prop}

\begin{proof}
For contradiction, suppose 
$${\mathcal{Z}(\sigma)}=\mathcal{Z}(d\Phi)$$
for some holomorphic function 
$\Phi$ on a nonsingular analytic variety $A$.
Let $\mathcal Z$ denote $\mathcal{Z}(\sigma)=\mathcal{Z}(d\Phi)$.

We show $A$ can be cut down to 2 dimensions.
Since $\mathcal{Z}$ has a 2-dimensional Zariski tangent space, 
$A$ has a product structure (perhaps after shrinking to a Euclidean open neighbourhood) 
\begin{equation} \label{ABC}
A=B\times C
\end{equation}
with $B$ nonsingular of dimension 2 and 
$$\mathcal{Z}\subset B\times\{c\}\subset B\times C=A\,.$$
Let $b\in B$ be the point at which $\mathcal{Z}$ is supported, and
$$p=(b,c) \in A\,.$$

Now $D(d\Phi)|_p$ is injective on $T_{c}C$ because $T_{c}C$ 
is a complement to the kernel $T_bB$. 
Since $D(d\Phi)|_p$ is symmetric,  
$\operatorname{Im}D(d\Phi)|_p$ lies
 in the annihilator of
$\operatorname{ker}D(d\Phi)|_p$, which is $\Omega_C|_{c}$. Thus,
 the composition
$$
\xymatrix@C=30pt{
T_C\ \ar@{^(->}[r] & T_A \rto^(.48){D(d\Phi)} & \Omega_A \ar@{->>}[r] & \Omega_C}
$$
on $\mathcal{Z}$ 
is injective at $p$. The composition is therefore an isomorphism at $p$, and hence, by openness, an isomorphism in a neighbourhood of $\mathcal{Z}$.

Thus, writing $(d_B\Phi,d_C\Phi)$ for $d\Phi=(d_{B\times C/C}\Phi,
d_{B\times C/B}\Phi)$ in the product structure \eqref{ABC}, we find that the zero locus
$$
B'=\mathcal Z(d_C\Phi)
$$
of $d_C\Phi$ is tangent to $B$ at $p$ and smooth and 2-dimensional in a neighbourhood. Shrinking if necessary, we can assume that $B'$ is everywhere nonsingular and never tangent to the $C$ fibres of $A=B\times C$.

Now $\mathcal Z$ is the zero locus of $d_B\Phi$ on $B'$. By the tangency condition, $\mathcal Z$ is the same as the zero locus of $(d\Phi)|_{B'}=d(\Phi|_{B'})$.
Thus $\mathcal Z$  is the critical locus of $\Phi|_{B'}$, with $B'$
nonsingular and 2-dimensional, contradicting  Theorem \ref{Zurich}.
\end{proof}

\vspace{+8 pt}

\noindent Departement Mathematik \hfill Department of Mathematics \\
\noindent ETH Z\"urich \hfill  Princeton University \\
\noindent rahul@math.ethz.ch  \hfill rahulp@math.princeton.edu \\

\vspace{+8 pt}
\noindent
Department of Mathematics \\
Imperial College \\
richard.thomas@imperial.ac.uk

\end{document}